\documentclass[11pt]{amsart}
\usepackage{}
\usepackage[T1]{fontenc}
\usepackage{lmodern}
\usepackage{ifthen}
\usepackage{amsfonts}
\usepackage{amsxtra}
\usepackage{amssymb}
\usepackage{amsthm}
\usepackage{array}
\usepackage[margin=1in]{geometry}
\usepackage{xcolor}
\definecolor{cite}{rgb}{0.50,0.00,1.00}
\definecolor{url}{rgb}{0.00,0.50,0.75}
\definecolor{link}{rgb}{0.00,0.00,0.50}
\usepackage[colorlinks,linkcolor=link,urlcolor=url,citecolor=cite,pagebackref,breaklinks]{hyperref}
\usepackage{mathtools}
\usepackage{mathrsfs}
\usepackage[all]{xy}
\usepackage[lite,abbrev,msc-links]{amsrefs}

\makeindex

\theoremstyle{definition} 
\newtheorem{Unity}{Unity}[section] 
\newtheorem*{defn*}{Definition} 
\newtheorem{defn}[Unity]{Definition} 

\theoremstyle{plain} 
\newtheorem*{thm*}{Theorem}
\newtheorem{thm}[Unity]{Theorem}
\newtheorem{prop}[Unity]{Proposition}
\newtheorem*{cor*}{Corollary}
\newtheorem{cor}[Unity]{Corollary}
\newtheorem{lem}[Unity]{Lemma}
\newtheorem{conj}[Unity]{Conjecture}

\theoremstyle{remark} 
\newtheorem*{rmk*}{Remark}

\numberwithin{Unity}{section}

\begin{document}
\title{exterior power of stable vector bundle destabilized by Frobenius pull-back}
\author[Yongming Zhang]{Yongming Zhang}
\email{zhangym97@mail.sysu.edu.cn}
\address{School of Science, Sun Yat-sen University, Shenzhen, 518107, P. R. of China}
\maketitle
\begin{abstract}
 In this paper, we prove that for any smooth projective curve $C$ of genus $g\geq2$ over an algebraically closed field of positive characteristic, there exists a stable vector bundle over $C$ whose exterior power is not semi-stable.
\end{abstract}

\section{Introduction}
Let $C$ be a smooth irreducible projective curve of genus $g>0$ over an algebraically closed field $k$.
As we know, when ${\rm char} (k)=0$, the semi-stability of vector bundles over $C$ is preserved under the operations of  tensor product, symmetric product and exterior product operations (cf.  \cite{H71,G79}).
However, this is not the case in positive characteristic. In fact, in \cite{G73} the author proves that, for any prime $p$ and integer $g>1$, there exists a smooth curve $C$ of genus $g$ in characteristic $p$ and a semi-stable bundle $E$ of rank $2$ over $C$ such that the Frobenius pullback $E^{(p)}:=F^*E$ and thus the symmetric power $S^p(E)$ are not semi-stable. The examples of non-ample semi-stable bundles constructed by J.-P. Serre (cf. \cite[Sect. 3]{H71}) and generally by H. Tango (cf. \cite[Sect. 5, Exam. 3]{T72}), as well as the push-forward $F_*E$ of a vector bundle $E$, possess the required property. For the case of the exterior product, we found no relevant assertions; therefore, in this note, we present a similar conclusion regarding the exterior power of semi-stable bundles in positive characteristic.

First, we adopt a computational approach to obtain the following result.
\begin{thm}(Theorem \ref{unstablity})
Let $C$ be a smooth projective curve of genus $g\geq2$ over an algebraically closed field $k$ of characteristic $p>0$, and $E$ be a vector bundle of rank $r$ on $C$. Then $F^n_*E\wedge F^n_*E$ is not semi-stable whenever $r>1$, $p>3$ or $n>1$.
\end{thm}
Therefore, as a corollary, we derive a similar result in case of exterior power of a stable vector bundle.
\begin{thm}(Corollary \ref{unstablity})
Let $C$ be a smooth projective curve of genus $g\geq2$ defined over an algebraically closed field $k$ of characteristic $p>0$, then there exists a stable vector bundle $E$ on $C$, whose double exterior power is not semi-stable.
\end{thm}

By definition, cohomological stability implies stability in the usual sense.In characteristic zero, any exterior power of a semi-stable bundle is also semi-stable, so cohomological semi-stability is equivalent to semi-stability.However,
in positive characteristic, since exterior product no longer preserves semi-stability, cohomological semi-stability is not equivalent to slope semi-stability. So we intended to find counterexamples of vector bundles that are semi-stable but not cohomologically semi-stable and we only confirm that stability is not equivalent to cohomological stability in positive characteristic.
\begin{prop}(Proposition \ref{costablity})
Let $C$ be a smooth projective curve of genus $g\geq2$ defined over an algebraically closed field $k$ of characteristic $p>0$, then there exists a vector bundle  $E$ which is stable but not cohomologically stable.
\end{prop}
Based on the above observation, we propose the following conjecture.
\begin{conj}
Let $C$ be a smooth projective curve defined over an algebraically closed field $k$ of characteristic $p>0$, and $E$ be a vector bundle over $C$. Then $E$ is semi-stable if and only if $E$ is cohomologically semi-stable.
\end{conj}
\textbf{Acknowledgement:} The author would like to thank Yi Gu and Lingguang Li for their valuable discussions concerning this work.

\section{Preliminary}
In this section, we recall some definitions and facts about the stability of vector bundles on curves.
\subsection{Notation}
Let $C$ be a smooth projective curve of genus $g>1$ over an algebraically closed field $k$ of positive characteristic and $\Omega_C$ be the sheaf of $1$-forms on $C$. The absolute Frobenius morphism $F: C\rightarrow C$ is defined by the map $f\rightarrow f^{p}$ in $\mathcal{O}_C$, and the $n$-th Frobenius morphism $F^n$ is $n$-times composition of $F$, where $n\geq0$. If $n=0$, we mean $F^0=Id$. Let $B_1$ denote the exact $1$-form which sits in the following exact sequence
$$0\rightarrow\mathcal{O}_C\stackrel{F^\#}{\rightarrow}F_*\mathcal{O}_C\rightarrow B_1\rightarrow0.$$
and is semi-stable of degree $(p-1)(g-1)$.
\subsection{Stability}
\begin{defn}\label{def}
Let $E$ be a vector bundle on a smooth projective curve $C$. Recall that the \emph{slope} of $E$ is the rational number $$\mu(E)=\frac{\deg E}{\text{rk}(E)},$$ and $E$ is called \emph{stable} (resp. \emph{semi-stable}) if $$\mu(F)<\mu(E)\ (\text{resp. } \mu(F)\leq\mu(E))$$ for every sub-bundle $F\subset E$ with $\text{rk}(F)<\text{rk}(E)$.

Moreover, $E$ is called \emph{cohomologically stable} (resp. \emph{cohomologically semi-stable}) if for every line bundle $A$ of degree $a$ and every integer $t<\text{rk}(E)$, $$H^0(\wedge^tE\otimes A^{-1})=0$$ whenever $a\geq t\mu(E)$ (resp. $a>t\mu(E)$).
\end{defn}
Let $E$ be a vector bundle on $C$. Then there exists a unique filtration,  which is called the
\emph{Harder-Narasimhan filtration},
$$0=E_0\subset E_1\subset\cdots\subset E_s=E$$
such that $E_i/E_{i-1}$ is semi-stable of slope $\mu_i$ and
$$\mu_1>\mu_2>\cdots>\mu_s.$$
\begin{prop}\label{prop2.2}
Let $E$ be a vector bundle on $C$. Then
$$\mu(F^n_*(E))=\frac{\mu(E)}{p^n}+(1-\frac{1}{p^n})(g-1).$$
\end{prop}
\begin{proof}
Note that $\chi(F^n_*E)=\chi(E)$, then by the Riemann-Roch theorem we have
$\deg(F^n_*E)=\deg(E)+\text{rk}(E)(p^n-1)(g-1)$ and hence the result.
\end{proof}
\subsection{A canonical filtration}
Let $E$ be a vector bundle on $C$. Then $V:=F^*(F_*E)$ has a canonical
connection $$\nabla: V\rightarrow V\otimes\Omega_C$$ with zero $p$-curvature (cf. \cite[Theorem 5.1]{K70}).
So there is a canonical filtration
$$0=V_p \subset V_{p-1}\subset\cdots\subset V_1\subset V_0=V\eqno{\#}$$
where $V_1=\rm{ker} (V=F^*F_*E \twoheadrightarrow E)$ and
$$V_{l+1}={\rm ker} (V_{l}\stackrel{\nabla}{\rightarrow}V\otimes\Omega_C\twoheadrightarrow V/V_l\otimes\Omega_C )$$
when $l>0$.
Then we have the following lemma by a local calculation (cf.\cite[Lem 2.1]{Sun08}).
\begin{lem}\label{lem2.3}(\cite[Theorem 5.3]{J06})
Under the above assumption, we have
\begin{enumerate}
  \item $V_0/V_1\cong E,\nabla(V_{l+1})\subset V_l\otimes\Omega_C$ for $l>0$,
  \item $V_l/V_{l+1}\stackrel{\nabla}{\rightarrow}(V_{l-1}/V_l)\otimes\Omega_C$ is an isomorphism for $1\leq l\leq p-1$, and
  \item if $g\geq 2$ and $E$ is semi-stable, then the canonical filtration $(\#)$ is nothing but the Harder-Narasimhan filtration.
\end{enumerate}
\end{lem}
This lemma can be used to prove the following result (see \cite{LP08} for the case of line bundles).
\begin{lem}(\cite[Theorem 2.2]{Sun08})\label{2.4}
Let $C$ be a smooth projective curve of genus $g\geq1$. Then $F_*E$ is semi-stable whenever $E$ is semi-stable. If $g\geq2$, then $F_*E$ is stable
whenever $E$ is stable.
\end{lem}
\subsection{Combination number modulo p}

\begin{lem}\label{2.5}
Let $p>0$ be a prime number. Then we have
\begin{itemize}
  \item ${p-1\choose h}\equiv(-1)^h \ \mathrm{mod}(p)$, and
  \item ${p-2\choose h}\equiv(-1)^h(h+1) \ \mathrm{ mod}(p)$.
\end{itemize}
\end{lem}
\begin{proof}
Note that $${p-1\choose h}=\frac{(p-1)!}{(p-1-h)!h!}={p-1\choose h-1}\frac{p-h}{h}\equiv-{p-1\choose h-1}\ {\rm mod}(p)$$ and $${p-2\choose h}=\frac{(p-2)!}{(p-2-h)!h!}={p-2\choose h-1}\frac{p-h-1}{h}\equiv-\frac{h+1}{h}{p-2\choose h-1}\ {\rm mod}(p).$$ Then we obtain the result by iteration.
\end{proof}
\section{Instability of $F^n_*E\wedge F^n_*E$}
In this section,  we first make the following observation: $F_*^n(E)\wedge F_*^n(E)$ admits a canonical sub-bundle. We then utilize this observation to derive our main result by computing the slopes of the bundles involved in the subsequent assertions.
\begin{lem}\label{keylem}
Let $C$ be a smooth projective curve defined over an algebraically closed field $k$ of characteristic $p>0$, and $E$ be a vector bundle of rank $r$ on $C$. Then
\begin{enumerate}
  \item when $r>1$, there is a natural injection
  $$F_*^n(E\wedge E\otimes \Omega_C^{p^n-1})\hookrightarrow F_*^n(E)\wedge F_*^n(E),$$
  where $n\in \mathbb{Z}^+$;
  \item when $r=1$ and $p>2$, there is a natural injection
  $$F_*^n(E\otimes E\otimes \Omega_C^{p^n-2})\hookrightarrow F_*^n(E)\wedge F_*^n(E),$$
  where $n\in \mathbb{Z}^+$;
  \item when $r=1$ and $p=2$, there is a natural injection
  $$F_*^{n-1}(B_1\otimes E\otimes \Omega_C^{p^{n-1}-1})\hookrightarrow F_*^n(E)\wedge F_*^n(E),$$
  where $n\in \mathbb{Z}^+$.
\end{enumerate}
\end{lem}
\begin{proof}
By Lemma \ref{lem2.3} there is a canonical filtration
$$0= V_p\subset\cdots\subset V_1\subset V_0=F^*F_*(E)\eqno(*)$$ with $V_i/V_{i+1}\cong E\otimes\Omega_C^i$.
After tensoring the above filtration with $E$ and pushing it forward by $F_*$,  by projection formula we obtain the following filtration of $F_*E\otimes F_*E$:
$$0= F_*(E\otimes V_p)\subset\cdots\subset F_*(E\otimes  V_1)\subset F_*(E\otimes V_0)\cong F_*E\otimes F_*E $$ with $F_*(E\otimes V_i )/F_*(E\otimes V_{i+1})\cong  F_*(E\otimes E\otimes \Omega_C^i)$.

Next, we give a local basis of the first non-zero sub-bundle $F_*(E\otimes V_{p-1})\cong F_*(E\otimes E\otimes \Omega_C^{p-1})$. Since it is a local problem,
we first reduce to the case $E=\mathcal{O}_C$ and $C=\rm {Spec}\, k[[t]]$, then $F^*F_*\mathcal{O}_C=k[[t]]\otimes_{k[[s]]}k[[t]]$ where $s=t^p$. Set $I_0=k[[t]]\otimes_{k[[s]]}k[[t]]$ and $I_1=\ker(k[[t]]\otimes_{k[[s]]}k[[t]]\twoheadrightarrow k[[t]])$ which is also an ideal of the $k[[t]]$-algebra $I_0=k[[x]]\otimes_{k[[s]]}k[[x]]$.
Set $\{\alpha:=1\otimes t-t\otimes1\}$, and it is easy to see that $\{\alpha,\alpha^2,\ldots,\alpha^{p-1}\}$ is a basis of $I_1$ as $k[[t]]$-module ($\alpha^p=0$). Note that the filtration ($*$) is defined by $I_{l+1}=\ker(I_l\stackrel{\nabla}{\rightarrow}I_0\otimes \Omega_C\twoheadrightarrow I_0/I_l\otimes \Omega_C)$ and
$\nabla(\alpha^l)=-l\alpha^{l-1}\otimes \mathrm{d}t$.
Thus, as a free $\mathcal{O}$-module, $I_l$ has a basis $\{\alpha^l,\alpha^{l+1},\ldots,\alpha^{p-1}\}$.

Now, we return to the case when $r>1$ and suppose that locally $E=\mathcal{O}\cdot e_1\oplus\cdots\oplus\mathcal{O}\cdot e_r$. Then the local basis of $F_*(E\otimes E\otimes \Omega_C^{p-1})$ consists of:
\begin{eqnarray*}
     \{(e_i\otimes e_j)\cdot t^k\alpha^{p-1}&:=&\sum_{h=0}^{p-1}{p-1\choose h}(-1)^ht^{k+h}e_i\otimes t^{p-1-h}e_j\\
     &=&\sum_{h=0}^{p-1}t^{k+h}e_i\otimes t^{p-1-h}e_j,\, i, j=1,2,\ldots,r, k=0,1,\ldots,p-1\}
\end{eqnarray*}
since ${p-1\choose h}\equiv(-1)^h \ \mathrm{mod}(p)$ by Lemma \ref{2.5}.
Note that
\begin{eqnarray*}
t^k\alpha^{p-1}&=&\sum_{h=0}^{p-1}t^{k+h}\otimes t^{p-1-h}=\sum_{i=k}^{p-1+k}t^{i}\otimes t^{p-1+k-i}=\sum_{i=k}^{p-1}t^{i}\otimes t^{p-1+k-i}+\sum_{i=p}^{p-1+k}t^{i}\otimes t^{p-1+k-i}\\
&=&\sum_{i=k}^{p-1}t^{i}\otimes t^{p-1+k-i}+\sum_{j=0}^{k-1}t^{p+j}\otimes t^{k-1-j}=\sum_{i=k}^{p-1}t^{i}\otimes t^{p-1+k-i}+s\sum_{j=0}^{k-1}t^{j}\otimes t^{k-1-j},
\end{eqnarray*}
which implies that $t^k\alpha^{p-1}, k=0,1,\ldots,p-1$ are all symmetric (if $m>n$, we mean $\Sigma_{i=m}^{n}a_i=0$).
So the symmetric part of $F_*(E\otimes E\otimes \Omega_C^{p-1})$ is generated by
$$(e_i\otimes e_i)\cdot t^k\alpha^{p-1}, i=1,2,\ldots,r, k=0,1,\ldots,p-1, \text{ and }$$
$$ (e_i\otimes e_j+e_j\otimes e_i)\cdot t^k\alpha^{p-1}, i\neq j, k=0,1,\ldots,p-1,$$
which form a basis of the kernel of the natural map:
$F_*(E\otimes E\otimes \Omega_C^{p-1})\twoheadrightarrow F_*(E\wedge E\otimes \Omega_C^{p-1})$ as $k[[s]]$-module. Therefore, we obtain an injection
$$F_*(E\wedge E\otimes \Omega_C^{p-1})\hookrightarrow F_*(E)\wedge F_*(E).$$

Moreover, by an $n$-times composition of the above map, we get the desired conclusion:
$$F_*^n(E\wedge E\otimes \Omega_C^{p^n-1})\hookrightarrow F_*^n(E)\wedge F_*^n(E).$$

When $r=1$, by the above argument, we see that the sections in the sub-bundle $F_*(E\otimes V_{p-1})\subset F_*(E\otimes V_{p-2})$ are all symmetric. So we consider the second non-zero sub-bundle $F_*(E\otimes V_{p-2})\subset F_*E\wedge F_*E$ and hence the local basis of the quotient $F_*(E\otimes V_{p-2})/ F_*(E\otimes V_{p-1})\cong  F_*(E\otimes E\otimes \Omega_C^{p-2})$. Suppose that locally $E=\mathcal{O}\cdot e$. As the above argument, the local basis of $F_*(E\otimes E\otimes \Omega_C^{p-2})$ consists of:
\begin{eqnarray*}
     \{(e\otimes e)\cdot t^k\alpha^{p-2}&:=&\sum_{h=0}^{p-2}{p-2\choose h}(-1)^ht^{k+h}e\otimes t^{p-2-h}e\\
     &=&\sum_{h=0}^{p-2}(h+1)t^{k+h}e\otimes t^{p-2-h}e,\, k=0,1,\ldots,p-1\}
\end{eqnarray*}
since ${p-2\choose h}\equiv(-1)^h(h+1) \ \mathrm{ mod}(p)$ by Lemma \ref{2.5}.
Note that
\begin{eqnarray*}
t^k\alpha^{p-2}&=&\sum_{h=0}^{p-2}(h+1)t^{k+h}\otimes t^{p-2-h}=\sum_{i=k}^{p-2+k}(i-k+1)t^{i}\otimes t^{p-2+k-i}\\
&=&\sum_{i=k}^{p-2}(i-k+1)t^{i}\otimes t^{p-2+k-i}+(p-k)t^{p-1}\otimes t^{k-1}+\sum_{i=p}^{p-2+k}(i-k+1)t^{i}\otimes t^{p-2+k-i}\\
&=&\sum_{i=k}^{p-2}(i-k+1)t^{i}\otimes t^{p-2+k-i}+(p-k)t^{p-1}\otimes t^{k-1}+\sum_{j=0}^{k-2}(j+p-k+1)t^{j+p}\otimes t^{k-2-j}\\
&=&\sum_{i=k}^{p-2}(i-k+1)t^{i}\otimes t^{p-2+k-i}+(p-k)t^{p-1}\otimes t^{k-1}+s\sum_{j=0}^{k-2}(j-k+1)t^{j}\otimes t^{k-2-j},
\end{eqnarray*}
which implies that $t^k\alpha^{p-2}$ is symmetric if and only if $k=0$ and $p=2$.
(Here, if $m>n$ and $a<0$, we mean $\Sigma_{i=m}^{n}a_i=0$ and $t^b\otimes t^{a}=0$ respectively.)
So when $p>2$ the symmetric part of $F_*(E\otimes V_{p-2})$ is exactly the sub-bundle $F_*(E\otimes E\otimes \Omega_C^{p-1})$. Therefore, we obtain an injection
$$F_*(E\otimes E\otimes \Omega_C^{p-2})\hookrightarrow F_*(E)\wedge F_*(E).$$
Moreover, by considering $E:=F_* E$ which is of $\rm {rank}>1 $ in (1), we have
$$F_*^n(E\otimes E\otimes \Omega_C^{p^n-2})\hookrightarrow F_*^n(E)\wedge F_*^n(E).$$
When $p=2$, note that $F_*(E)\wedge F_*(E)\cong B_1\otimes E$ and by (1) again we have the required injection:
$$F_*^{n-1}(B_1\otimes E\otimes \Omega_C^{p^{n-1}-1})\hookrightarrow F_*^n(E)\wedge F_*^n(E)$$
\end{proof}

\begin{thm}
Let $C$ be a smooth projective curve of genus $g\geq2$ defined over an algebraically closed field $k$ of characteristic $p>0$, and $E$ be a vector bundle of rank $r$ on $C$ then $F^n_*E\wedge F^n_*E$ is not semi-stable when $r>1$, $p>3$ or $n>1$.
\end{thm}
\begin{proof}
When $r>1$, by Lemma \ref{keylem}, we have the following injection
$$F^n_*(E\wedge E\otimes \Omega_C^{p^n-1})\hookrightarrow F_*^n(E)\wedge F_*^n(E).$$
After some computation of the slopes by using Proposition \ref{prop2.2}, we have:
$$\mu(F_*^nE\wedge F_*^nE)=2\mu(F_*^nE)=\frac{2}{p^n}\mu(E)+\frac{2(p^n-1)(g-1)}{p^n}$$ and
\begin{eqnarray*}
\mu(F_*^n(E\wedge E\otimes \Omega_C^{p^n-1}))&=&\frac{1}{p^n}\mu(E\wedge E\otimes \Omega_C^{p^n-1})+\frac{(p^n-1)(g-1)}{p^n}\\
&=&\frac{1}{p^n}(2\mu(E)+\mu(\Omega_C^{p^n-1}))+\frac{(p^n-1)(g-1)}{p^n}\\
 &=&\frac{2}{p^n}\mu(E)+\frac{3(p^n-1)(g-1)}{p^n}\\
 &>&\mu(F_*^nE\wedge F_*^nE).
\end{eqnarray*}
So it is not semi-stable.

When $r=1$ and $p>2$,
by Lemma \ref{keylem}, we have the following injection
$$F_*^n(E\otimes E\otimes \Omega_C^{p^n-2})\hookrightarrow F^n_*(E)\wedge F^n_*(E).$$
After some computation of the slopes by using Proposition \ref{prop2.2}, we have:
$$\mu(F_*^nE\wedge F_*^nE)=2\mu(F_*^nE)=\frac{2}{p^n}\mu(E)+\frac{2(p^n-1)(g-1)}{p^n}$$ and
\begin{eqnarray*}
\mu(F_*^n(E\otimes E\otimes \Omega_C^{p^n-2}))&=&\frac{1}{p^n}\mu(E\otimes E\otimes \Omega_C^{p^n-2})+\frac{(p^n-1)(g-1)}{p^n}\\
&=&\frac{1}{p^n}(2\mu(E)+\mu(\Omega_C^{p^n-2}))+\frac{(p^n-1)(g-1)}{p^n}\\
 &=&\frac{2}{p^n}\mu(E)+\frac{(3p^n-5)(g-1)}{p^n}\\
 &>&\mu(F^n_*E\wedge F^n_*E),
\end{eqnarray*} when $p>3$ or $p=3$ and $n>1$.
So $F^n_*E\wedge F^n_*E$ is not semi-stable when $p>3$ or $p=3$ and $n>1$.

When $r=1$ and $p=2$,
by Lemma \ref{keylem} we have the following injection
$$F_*^{n-1}(B_1\otimes E\otimes \Omega_C^{p^{n-1}-1})\hookrightarrow F^n_*(E)\wedge F^n_*(E).$$
After some computation of the slopes by using Proposition \ref{prop2.2}, we have:
$$\mu(F_*^nE\wedge F_*^nE)=2\mu(F_*^nE)=\frac{2}{p^n}\mu(E)+\frac{2(p^n-1)(g-1)}{p^n}$$ and
\begin{eqnarray*}
\mu(F_*^{n-1}(B_1\otimes E\otimes \Omega_C^{p^{n-1}-1}))&=&\frac{1}{p^{n-1}}\mu(B_1\otimes E\otimes \Omega_C^{p^{n-1}-1})+\frac{(p^{n-1}-1)(g-1)}{p^{n-1}}\\
&=&\frac{1}{p^{n-1}}(\mu(B_1)+\mu(E)+\mu(\Omega_C^{p^{n-1}-1}))+\frac{(p^{n-1}-1)(g-1)}{p^{n-1}}\\
 &=&\frac{2}{p^n}\mu(E)+\frac{(3p^n-4)(g-1)}{p^n}\\
 &>&\mu(F^n_*E\wedge F^n_*E),
\end{eqnarray*} when $n>1$.
So $F^n_*E\wedge F^n_*E$ is not semi-stable when $n>1$.
\end{proof}
\begin{cor}\label{unstablity}
Let $C$ be a smooth projective curve of genus $g\geq2$ defined over an algebraically closed field $k$ of characteristic $p>0$, then there exists a stable vector bundle $E$ on $C$, whose double exterior product is not semi-stable.
\end{cor}
\begin{proof}
Take a line bundle $L$ on $C$, then $E:=F^n_*L$ is stable by Lemma \ref{2.4}.
Thus by the above theorem we see that $E\wedge E$ is not semi-stable when $n>1$.
\end{proof}
\begin{prop}\label{costablity}
Let $C$ be a smooth projective curve of genus $g\geq2$ defined over an algebraically closed field $k$ of characteristic $p>0$, then there exists a vector bundle  $E$ which is stable but not cohomologically stable.
\end{prop}
\begin{proof}
Take $E=F^n_*L$ for some line bundle $L$ of degree $d$ and some integer $n>1$, then $E$ is stable by Lemma \ref{2.4}.
By Definition \ref{def}, we must prove that there is an integer $t<rk(E)$ and a  line bundle $A$ of degree $a\geq t\mu(E)$ satisfying $H^0(\wedge^tE_L\otimes A^*)\neq0$
When $p=2$ by Lemma \ref{keylem} we have
$$F^{n-1}_*(B_1\otimes L\otimes\Omega_C^{p^{n-1}-1})\hookrightarrow F_*^nL\wedge F_*^nL.$$
We take the line bundle $L$ of degree $d$ such that $d+(p^n-1)(g-1)$ can be divided by $p^{n-1}$ and then take a line bundle $A$ of degree $\frac{d+(p^n-1)(g-1)}{p^{n-1}}$ such that $A^{ p^{n-1}}\cong B_1\otimes L\otimes\Omega_C^{p^{n-1}-1}$. Thus we have $\deg A\geq 2\mu(F^n_*L)=\frac{d+(p^n-1)(g-1)}{p^{n-1}}$, but $h^0(F^n_*L\wedge F^n_*L\otimes A^{-1})\geq h^0(F^{n-1}_*(B_1\otimes L\otimes\Omega_C^{p^{n-1}-1})\otimes A^{-1})=h^0(F^{n-1}_*(B_1\otimes L\otimes\Omega_C^{p^{n-1}-1}\otimes A^{-p^{n-1}}))=h^0(\mathcal{O}_C)\neq0$, which implies that $F_*^nL$ is not cohomologically stable.

When $p>2$ by Lemma \ref{keylem} we have
$$F^{n}_*(L\otimes L\otimes\Omega_C^{p^{n}-1})\hookrightarrow F_*^nL\wedge F_*^nL.$$
We take the line bundle $L$ of degree $d$ such that $2d+2(p^n-1)(g-1)$ can be divided by $p^{n}$ and then take a line bundle $A$ of degree $\frac{2d+2(p^n-1)(g-1)}{p^{n}}$ such that $A^{ p^{n}}\cong L\otimes L\otimes\Omega_C^{p^{n}-1}$. Thus we have $\deg A\geq 2\mu(F^n_*L)=\frac{2d+2(p^n-1)(g-1)}{p^{n}}$. But $h^0(F^n_*L\wedge F^n_*L\otimes A^{-1})\geq h^0(F^{n}_*(L\otimes L\otimes\Omega_C^{p^{n}-1})\otimes A^{-1})=h^0(F^{n}_*(L\otimes L\otimes\Omega_C^{p^{n}-1}\otimes A^{-p^{n}}))=h^0(\mathcal{O}_C)\neq0$, which implies that $F_*^nL$ is not cohomologically stable.
\end{proof}

\end{document}